\documentclass{amsart}
\usepackage{fancyhdr}
\usepackage{amsthm}
\usepackage{amssymb}
\usepackage{amscd}
\usepackage{amsmath, pb-diagram}
\usepackage[latin5]{inputenc}
\usepackage[all]{xy}
\usepackage{setspace}
\usepackage{graphics,epsfig}
\usepackage{graphicx}
\usepackage{float}
\usepackage{array}
\usepackage{xfrac}
\usepackage{faktor}
\usepackage{multicol}
\usepackage[shortlabels]{enumitem}
\usepackage[mathscr]{euscript}
\usepackage{mathtools}
\usepackage{tikz-cd}
\usepackage[margin=2.9cm]{geometry}
\usepackage{url}
\usepackage{nicefrac}
\usepackage{xfrac}

\allowdisplaybreaks

\theoremstyle{plain}
\newtheorem{theo}{Theorem}[section]
\newtheorem{cor}[theo]{Corollary}

\newtheorem{prop}[theo]{Proposition}

\theoremstyle{definition}
\newtheorem{defn}[theo]{Definition}
\newtheorem{ex}[theo]{Example}

\title{Multivariate Generalized Splines and Syzygies on Graphs}

\author{Selma Altinok \and Samet Sario\u{g}lan}

\address{Selma Altinok, Hacettepe University Department of Mathematics, 06800 Beytepe Ankara Turkey.}
\email{sbhupal@hacettepe.edu.tr}

\address{Samet Sario\u{g}lan (Corresponding author), Hacettepe University Department of Mathematics, 06800 Beytepe Ankara Turkey.}
\email{ssarioglan@hacettepe.edu.tr}

\begin{document}
\begin{abstract}
Given a graph $G$ whose edges are labeled by ideals of a commutative ring $R$ with identity, a generalized spline is a vertex labeling of $G$ by the elements of $R$ so that the difference of labels on adjacent vertices is an element of the corresponding edge ideal. The set of all generalized splines on a graph $G$ with base ring $R$ has a ring and an $R$-module structure. 

In this paper, we focus on the freeness of generalized spline modules over certain graphs  with the base ring $R = k[x_1 , \ldots , x_d]$ where $k$ is a field. We first show the freeness of generalized spline modules on graphs with no interior edges over $k[x,y]$  such as cycles or a disjoint union of cycles with free edges. Later, we consider graphs that can be  decomposed into disjoint cycles without changing the isomorphism class of the syzygy modules. Then we use this decomposition to show that generalized spline modules are free over $k[x , y]$ and later we extend this result to the base ring $R = k[x_1 , \ldots , x_d]$  under some restrictions. 
\end{abstract}

\maketitle

\section{Introduction}
\label{intro}

The concept of spline has been studied in two major approaches: Classical splines and generalized splines. Classical splines are collections of polynomials defined on the faces of a polyhedral complex that agree to a certain degree of smoothness on the intersection of faces. They are useful tools to control the curvature of objects in industry and have many applications related to numerical analysis, geometric design and solutions of partial differential equations. They form a module. For these and other applications, it is useful to study splines as a module.

Generalized splines are defined on edge labeled graphs. Given a finite graph $G=(V,E)$ and a commutative ring $R$ with identity, an edge labeling function $\alpha$ is a function that labels the edges of $G$ by the ideals of $R$. The pair $(G, \alpha)$ is called an edge labeled graph. A generalized spline on an edge labeled graph $(G, \alpha)$ is a vertex labeling $F \in R ^{|V|}$ such that for each edge $uv \in E$, the difference $f_u - f_v$ is an element of the ideal $\alpha(uv) \in R$ where $f_v$ denotes the label on vertex $v \in V$. The set of all generalized splines on $(G,\alpha)$ with base ring $R$, denoted by $R_{(G,\alpha)}$, has a ring and an $R$-module structure. If $R$ is a multivariate polynomial ring, we call such generalized splines as multivariate generalized splines. We use two notations for generalized splines: the column matrix notation with entries in order from bottom to top and the vector notation. Consider the edge labeled graph $(G,\alpha)$ in Figure~\ref{c3}.

\begin{figure}[H]
\centering
\scalebox{0.16}{\includegraphics{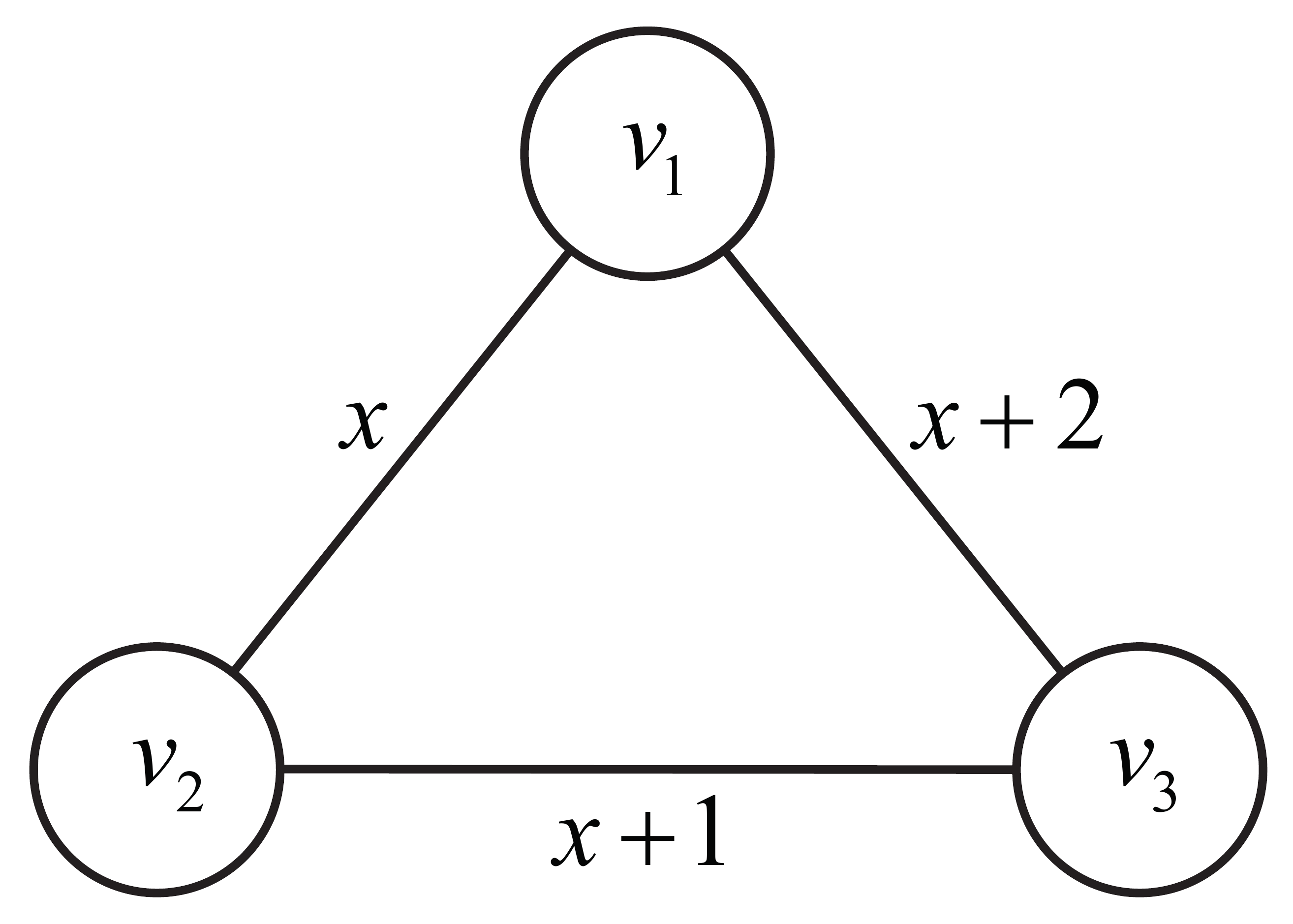}}
\caption{Edge labeled graph $(G,\alpha)$}
\label{c3}
\end{figure}

\noindent
A generalized spline on $(G,\alpha)$ can be presented by $F = \begin{bmatrix} x^2 + 2x + 1 \\ x+1 \\ 1 \end{bmatrix}$ or $F = (1, x+1 , x^2 + 2x + 1)$.

The motivation for studying generalized splines is based on algebraic geometry and topology. Studies in these areas showed that the ring structure of generalized splines corresponds to the equivariant cohomology rings of toric and other algebraic varieties~\cite{Gor, Sch2, Tym2}. Goresky and the others~\cite{Gor} studied the combinatorial structure of the equivariant cohomology ring corresponding to an algebraic variety X with a proper torus action. In this work, they defined the edge labeled graph $(G_X, \alpha)$ corresponds to the algebraic variety $X$ and showed that the equivariant cohomology ring of $X$ agrees with the ring structure of $R_{(G_X, \alpha)}$ where $R$ is a polynomial ring. Thus the algebraic structure of generalized splines over polynomial rings has importance. 

Our interest in the paper is to study the freeness of generalized splines as a module. In~\cite{Ros1}, Rose studied classical splines and observed that the module of splines on a polyhedral complex $\Delta$ can be viewed as a direct sum of the syzygy module of its dual graph with edges labeled by powers of linear forms. In~\cite{Rose}, by using a decomposition of an edge labeled graph $G$ without changing the isomorphism class of the syzygy module, Rose obtained results related to the freeness of the syzygy module generated by linear edge labels that meet certain conditions. We extend these ideas to generalized splines on arbitrary graphs which can be decomposed into disjoint cycles and free edges.

This paper consists of five sections. In Section~\ref{projdim}, we present basic definitions and properties related to projective modules and dimension. In Section~\ref{rankofacycle}, we discuss the rank of a cycle and the relation between rank and projective dimension. In Section~\ref{decomp}, we first introduce the module $M_v$, then we define the edge decomposition operations and finally we analyze the effects of these operations on the isomorphism class of the generalized spline module. The main results of this paper are presented in Section~\ref{freeness}. In this section, we first prove the freeness of $R_{(G,\alpha)}$ where $R = k[x,y]$ and $G$ is a cycle. Then we generalize this result to a graph containing only one cycle or having no interior edges. In the case of $R = k[x_1 , \ldots , x_d]$, we first give  freeness criteria for generalized spline modules on cycles and then we generalize our results to graphs that decompose into disjoint cycles and free edges under some conditions as the final result of the paper.

In the rest of the paper, we refer to multivariate generalized splines as splines.

\section{Projective Modules and Projective Dimension}
\label{projdim}

Let $P$ be an $R$-module. $P$ is said to be projective if for every surjective module homomorphism $f: A \to B$ and every module homomorphism $g: P \to B$, there exists a module homomorphism $h: P \to A$ such that the following diagram commutes, namely $f \circ h = g$.
\begin{displaymath}
\xymatrix{
{} & P \ar@{-->}[ld]_{h} \ar[d]^g & {}  \\
A  \ar[r]^f & B \ar[r] & 0}
\end{displaymath}

Every free module is projective, but the converse is not true in general. If $R$ is a principal ideal domain, then every projective $R$-module is free. Every finitely generated projective module over polynomial rings is also free by Quillen-Suslin Theorem. Basic properties of projective modules are given below.

\begin{prop} Let $P$ be an $R$-module. Then
\begin{enumerate}[label=\alph*)]
	\item $P$ is projective if and only if all short exact sequences $0 \to A \to B \to P \to 0$ of $R$-modules split.
	\item $P$ is projective if and only if it is a direct summand of a free $R$-module.
\end{enumerate}
 \label{projmod}
\end{prop}

\begin{proof} See Proposition 6.73 and Theorem 6.76 in~\cite{Rot}.
\end{proof}

\begin{prop} Let $\{ P_i \text{ $\vert$ } i \in I \}$ be a family of $R$-modules. Then $P = \bigoplus\limits_{i \in I} P_i$ is projective if and only if $P_i$ is projective for all $i \in I$.
\label{projmod2}
\end{prop}

\begin{proof} See Proposition 6.75 in~\cite{Rot}.
\end{proof}

Let $A$ be an $R$-module. A projective resolution of $A$ is an exact sequence
\begin{displaymath}
\mathcal{P}_{\bullet}: \ldots \to P_2 \to P_1 \to P_0 \to A \to 0
\end{displaymath}
in which each $P_n$ is projective. Every $R$-module $A$ has a projective resolution. The length of a shortest projective resolution of $A$ is called the projective dimension of $A$, denoted by $\text{pd }A$. If $A$ has no finite projective resolution, then $\text{pd }A = \infty$. An $R$-module $A$ is projective if and only if $\text{pd }A = 0$.

The following theorem is given without proof in~\cite{Pass} page 76.

\begin{theo} Let $0 \to B \to P \to A \to 0$ be a short exact sequence of $R$-modules with $P$ is projective. Then $\emph{pd }A = a > 0$ implies $\emph{pd }B = a-1$.
\label{pdim}
\end{theo}

\begin{proof}
Let $\text{pd }A = a > 0$ and $\text{pd }B = b$. We show that $a = b+1$. Since $\text{pd }B = b$, then there is a projective resolution of $B$ with length $b$ such as $0 \to P_b \to \ldots \to P_1 \to P_0 \to B \to 0$ and we get a finite projective resolution of $A$ with length $b+1$ by setting $0 \to P_b \to \ldots \to P_1 \to P_0 \to P \to A \to 0$. Hence we conclude that $a \leq b+1$.

Now assume that $a < b+1$. Let $0 \to Q_b \to \ldots \to Q_1 \to Q_0 \to B \to 0$ be a projective resolution of $B$ with length $b$. This complex breaks up into short exact sequences $0 \to B_{i+1} \to  Q_i \to B_i \to 0$ of $R$-modules with $B_0 = B$ and $B_{b+1} = 0$. We can do the same thing to a projective resolution of $A$ with length $a$. The resolution $0 \to P_a \to \ldots \to P_1 \to P_0 \to A \to 0$ breaks up into short exact sequences $0 \to A_{i+1} \to  P_i \to A_i \to 0$ of $R$-modules with $A_0 = A$ and $A_{a+1} = 0$. Applying Schanuel Lemma to the following short exact sequences
\begin{displaymath}
\begin{gathered}
0 \to B \to P \to A \to 0 \\
0 \to A_1 \to  P_0 \to A \to 0
\end{gathered}
\end{displaymath}
gives $P \oplus A_1 \cong P_0 \oplus B$. Now applying Schanuel Lemma to the following short exact sequences
\begin{displaymath}
\begin{gathered}
0 \to P \oplus B_1 \to P \oplus Q_0 \to P \oplus B \to 0 \\
0 \to P \oplus A_2 \to P \oplus P_1 \to P \oplus A_1 \to 0
\end{gathered}
\end{displaymath}
gives us $\underbrace{P_0 \oplus Q_0 \oplus P}_{\text{Projective}} \oplus A_2 \cong \underbrace{P \oplus P_1 \oplus P_0}_{\text{Projective}} \oplus B_1$.\bigskip

\noindent
Iterating this argument yields $\text{Projective} \oplus A_{j+1} \cong \text{Projective} \oplus B_j$. Here $A_a$ is projective, so that $\text{Projective} \oplus A_a$ is projective. Hence $B_{a-1}$ is projective by Theorem~\ref{projmod2}, so the resolution for $B$ ends at $B_{a-1}$ or before. Then $b \leq a-1$, which is a contradiction to our assumptation. Hence $a = b+1$.
\end{proof}

An upper bound for the projective dimension of modules over polynomial rings is given by Hilbert Syzygy Theorem.

\begin{theo} \emph{(Hilbert Syzygy Theorem)} Let $R = k[x_1 , \ldots , x_n]$ be the polynomial ring and $M$ be a finitely generated $R$-module. Then $\emph{pd }M \leq n$.
\label{hst}
\end{theo}

\begin{proof} See Corollary 10.167 in~\cite{Rot}.
\end{proof}

In order to give a useful property of the projective dimension, we need the concept of regular sequences. Given a ring $R$, an $R$-regular sequence is a finite sequence of elements $r_1 , \ldots , r_k \in R$ such that $r_i$ is not a zero divisor of $\nicefrac{R}  {\langle r_1 , \ldots , r_{i-1}} \rangle$ for each $i$ and $\langle r_1 , \ldots , r_k \rangle R \neq R$. Here $k$ is called the length of the sequence.

\begin{theo} Let $R$ be a commutative ring with identity and $I \subset R$ be an ideal generated by an $R$-regular sequence of length $n$. Then $\emph{pd }\nicefrac{R}{I} = n$.
\label{regtheo}
\end{theo}

\begin{proof} Let $I = \langle x_1 , \ldots , x_n \rangle$ where $x_1 , \ldots , x_n$ is an $R$-regular sequence. We use induction on $n$. For $n=1$, we have the following short exact sequence
\begin{equation}
0 \to R \xrightarrow{f} R \xrightarrow{\pi} \nicefrac{R}{<x_1>} \to 0
\label{eq0}
\end{equation}
such that $f(r) = x_1 r$ for all $r \in R$ and $\pi$ is the natural quotient map. By Equation~\ref{eq0}, we obtain $\text{pd }\nicefrac{R}{<x_1>} \leq 1$. Assume that $\text{pd }\nicefrac{R}{<x_1>}= 0$, then $\nicefrac{R}{<x_1>}$ is projective and so $<x_1>$ is a direct summand of $R$. Thus $x_1$ is an idempotent element, namely $x_1 ^2 = x_1$. Since $R$ is an integral domain, either $x_1 = 0$ or $x_1 = 1$. However both these cases are impossible since $x_1$ itself is an $R$-regular sequence. This contradicts our assumption. Hence we conclude that $\text{pd }\nicefrac{R}{<x_1>} = 1$.

For $n > 1$, we fix $R^* = \nicefrac{R}{<x_1>}$. Since $x_1 , \ldots , x_n$ is an $R$-regular sequence, $x_2 ^* , \ldots , x_n ^*$ is an $R^*$-sequence of length $n-1$. Hence we get $\text{pd }_{R^*} (\nicefrac{R^*}{<x_2 ^* , \ldots , x_n ^*>}) = \text{pd }_{R^*} (\nicefrac{R}{<x_1 , \ldots , x_n>}) = n-1$ by the induction hypothesis. Therefore we obtain
\begin{displaymath}
\text{pd }_{R} (\nicefrac{R}{<x_1 , \ldots , x_n>}) = 1 + \text{pd }_{R^*} (\nicefrac{R}{<x_1 , \ldots , x_n>}) = 1 + n-1 = n
\end{displaymath}
by the first change of rings theorem.
\end{proof}

\section{The Rank of a Cycle}
\label{rankofacycle}

Let $(C_n , \alpha)$ be an edge labeled cycle with edge labels $l_1 , \ldots , l_n$. The rank of $C_n$, denoted by $\text{rk } C_n$ is the dimension of the linear span $\langle l_1 , \ldots , l_n \rangle$.  In particular, the rank can  be seen as the codimension of the intersection of the hyperplanes $l_i = 0$ for all $i$ when all $l_i$ are linear.

Given an edge labeled cycle $(C_n , \alpha)$ with edge labels $\{ l_1 , \ldots , l_n \}$, fix the ideal $I = \langle l_1 , \ldots , l_n \rangle$. In the case all edge labels are homogeneous and linear, the relation between $\text{rk } C_n$ and $\text{pd } \nicefrac{R}{I}$ is given by the following theorem.

\begin{theo} Let $(C_n , \alpha)$ be an edge labeled cycle with all edge labels are linear and homogeneous. Then
\begin{displaymath}
\emph{pd } \nicefrac{R}{I} = \emph{rk } C_n.
\end{displaymath}
\label{ranktheo}
\end{theo}

\begin{proof} See Theorem 4.2 in~\cite{Ros1}.
\end{proof}

If the edge labels are not linear, then the statement of Theorem~\ref{ranktheo} does not hold. In this case, there is no even inequality between $\text{rk } C_n$ and $\text{pd } \nicefrac{R}{I}$. The following example illustrates this fact:

\begin{ex} Consider the following edge labeled $3$-cycles $(G_1 , \alpha_1)$ and $(G_2 , \alpha_2)$ with base rings $k[x,y,z]$ and $k[x,y,z,t]$ respectively.
\begin{figure}[H]
\centering
\scalebox{0.16}{\includegraphics{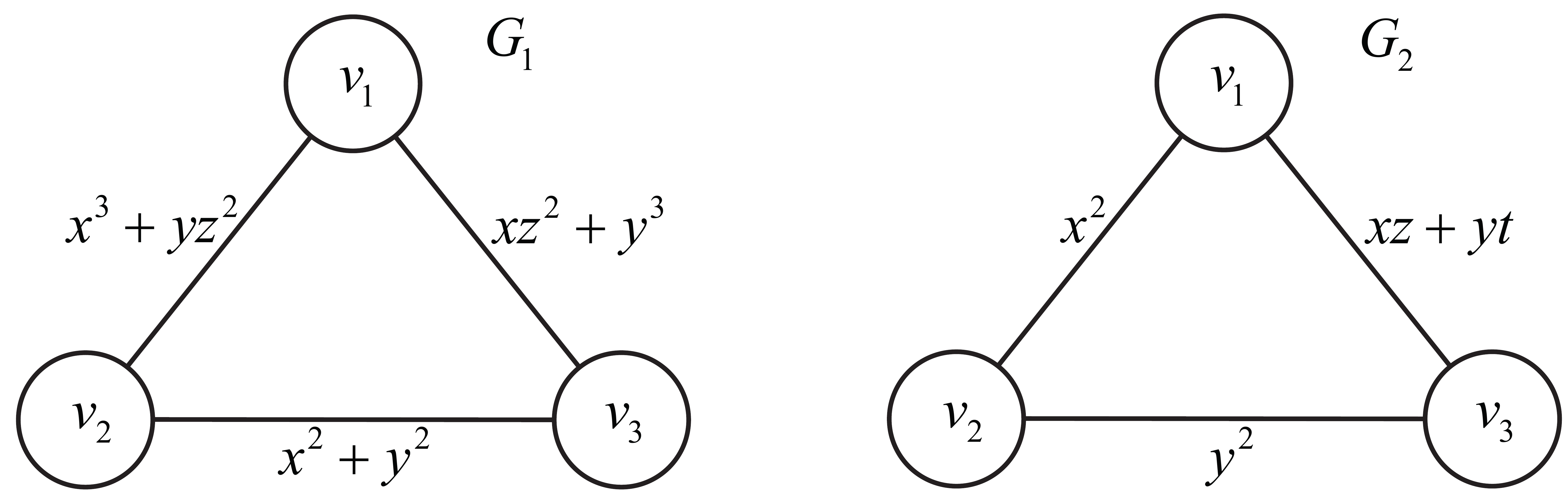}}
\caption{Relation between $\text{rk } C_3$ and $\text{pd } \nicefrac{R}{I}$}
\label{rk1}
\end{figure}

\noindent
Fix the ideals $I_1 = \langle x^3 + yz^2 , x^2 + y^2 , xz^2 + y^3 \rangle$ and $I_2 = \langle x^2 , y^2 , xz + yt \rangle$. It can be computed by Macaulay2~\cite{Gra} that $\text{rk } G_1 = 3$, $\text{pd } \big( \nicefrac{k[x,y,z]}{I_1} \big) = 2$ and $\text{rk } G_2 = 3$, $\text{pd } \big( \nicefrac{k[x,y,z,t]}{I_2} \big) = 4$.
\end{ex}

Even though Theorem~\ref{ranktheo} does not hold in general, we have the following theorem for cycles of rank at most two.

\begin{theo} Let $(C_n , \alpha)$ be an edge labeled cycle with rank at most two. Then $\emph{pd } \nicefrac{R}{I} \leq \emph{rk } C_n$.
\label{rkthm}
\end{theo}

\begin{proof} Let $(C_n , \alpha)$ be an edge labeled cycle with edge labels $\{ l_1 , \ldots , l_n \}$. Fix the ideal $I = \langle l_1 , \ldots , l_n \rangle$. We suppose that $\text{rk } C_n = 2$. Thus $I$ has a generating set with two elements, which may not be minimal. Without loss of generality, say $I = \langle l_1 , l_2 \rangle$. 

If $l_1, l_2$ are coprime, then $l_1, l_2$ is an $R$-regular sequence. In order to see this, assume that $l_1 ,l_2$  are coprime but  not an $R$-regular sequence. Hence $l_2$ is a zero divisor of $\nicefrac{R}{\langle l_1 \rangle}$, namely there exists an element $0 \neq r \in \nicefrac{R}{\langle l_1 \rangle}$ such that $r l_2 \in \langle l_1 \rangle$ and so  $l_1$ divides $r l_2$. Since $l_1,l_2$ are coprime  $l_1$  divides $r$ but this contradics the fact that $0\neq r\in \nicefrac{R}{\langle l_1\rangle}$. Therefore $l_1,l_2$ is a regular sequence so that by Theorem~\ref{regtheo}, $\text{pd }\nicefrac{R}{I} = 2$.

If a greatest common divisor $(l_1,l_2)\neq l_1 $  or $l_2$ then there exists a nonconstant common factor $r \in R$ satisfying $l_1 = f_1 r$ and $l_2 = f_2 r$ , where $f_1,f_2 \in R$  are coprime, so that $\langle f_1 , f_2 \rangle \cong \langle l_1 , l_2 \rangle$ as an $R$- module isomorphism via the map
\begin{align*}
\alpha : \langle f_1 , f_2 \rangle &\to \langle l_1 , l_2 \rangle \\
\alpha(f) &= rf.
\end{align*}
Hence $\text{pd }\nicefrac{R}{I} = \text{pd } \nicefrac{R}{\langle f_1 , f_2 \rangle} = 2$ as explained above.

If  we assume that without loss of generality $l_1=rl_2$, where $r\in R$ is a nonconstant, then $I = \langle l_1 , l_2 \rangle = \langle l_2 \rangle$. Since $R$ is an integral domain, $l_2$ is an $R$-regular sequence of length $1$ and hence $\text{pd }\nicefrac{R}{I} = 1$ by Theorem~\ref{regtheo}. Thus the inequality holds. In particular, when $r$ is constant or equivalenty $\text{rk } C_n = 1$ we have $\text{pd }\nicefrac{R}{I} =  1 = \text{rk } C_n$.

\end{proof}

\section{Decomposition of Edge Labeled Graphs}
\label{decomp}

In this section, we first present the module $M_v$ introduced by Tymoczko and the others in~\cite{Gil} to give a characterization of $R_{(G,\alpha)}$. Then we touch on decompositions of an edge labeled graph given by Rose in~\cite{Rose}.  

Let $(G,\alpha)$ be an edge labeled connected graph and fix a vertex $v \in V$. Then every spline $f \in R_{(G,\alpha)}$ can be expressed uniquely as $f = r \text{\textbf{1}} + f^v$, where $\text{\textbf{1}}$ is the trivial spline whose all entries are $1$, $f^v \in R_{(G,\alpha)}$ with $f^v _v = 0$ and $r = f_v$. In order to see this, define $f^v$ to be the spline $f^v = f - r \text{\textbf{1}}$. Hence $f^v _v = f_v - {r \text{\textbf{1}}}_v = r - r = 0$. This observation leads to the following theorem.

\begin{theo}~\cite{Gil} Let $(G,\alpha)$ be an edge labeled connected graph and $v \in V$. If $M_v = \langle f: f_v = 0 \rangle$ then $R_{(G,\alpha)} = R\emph{\textbf{1}} \oplus M_v$ as $R$-modules.
\label{mmodule}
\end{theo}

We can relate the module $M_v$ to the syzygy module generated by the edge labels as follows: Suppose that $G$ is the dual graph of a hereditary polyhedral complex $\Delta$ and $R = \mathbb{R}[x_1 , \ldots , x_n]$. Let $l_e$ be an affine form that generates the polynomials vanishing on the intersection of faces in $\Delta$ corresponding to each edge $e$ of $G$. Define the edge labeling function $\alpha$ as $\alpha(e) = l_e$ and let
\begin{displaymath}
B_{(G,\alpha)} = \left\{ (r_1 , \ldots , r_{|E|}) \in R^{|E|} : \sum_{e \in C} r_e l_e = 0 \text{ for all cycles } C \in G \right\}.
\end{displaymath}
Rose~\cite{Ros1} proved that $R_{(G,\alpha)} \cong R \oplus B_{(G,\alpha)}$ as $R$-modules. Together with Theorem~\ref{mmodule}, we conclude that $M_v \cong B_{(G,\alpha)}$ as $R$-modules for a fixed $v \in V$. We use $B_{(G,\alpha)}$ instead of $M_v$ for the rest of the paper.

Given an edge labeled graph $(G,\alpha)$, a cycle $C$ in $G$ that does not contain any smaller cycle is called a minimal cycle. The set of all minimal cycles in $G$ is denoted by $\mathcal{B}$. This set is also called a cycle space basis of $G$, see~\cite{Ian}. The syzygy module $B_{(G,\alpha)}$ can be presented as the kernel of the matrix $A$ as follows: The rows of $A$ are indexed by the elements of $\mathcal{B}$ and the columns are indexed by the edges of $G$. Then
\begin{displaymath}
A_{ij} = \begin{cases} 0, &\text{ if } e_j \notin C_i \\
\pm \alpha(e_j), &\text{ if } e_j \in C_i
\end{cases}.
\end{displaymath}
The sign of $\alpha(e_j)$ depends on the orientation of $e_j$ in $C_i$. If $G$ has no cycles, $A$ can be taken to be the row matrix $(0, \ldots , 0)$.

Let $e$ be an edge contained in two or more cycles in $G$. Our goal is to find conditions for deleting $e$ in one of the cycles, without changing the isomorphism class of $B_{(G,\alpha)}$.

\begin{defn} Let $(G,\alpha)$ be an edge labeled graph and $e \in E$. If $e$ is contained in two or more cycles, we call $e$ an interior edge. If $e$ is contained in only one cycle, we call $e$ an exterior edge, otherwise we call $e$ a free edge.
\end{defn}

We want to delete all interior edges from a cycle $C$ without changing the isomorphism class of $B_{(G,\alpha)}$ if it is possible. We give the following definition.

\begin{defn} Let $(G,\alpha)$ be an edge labeled graph and $e_i \in E$. If $e_i$ is an interior edge of a cycle $C \subset G$ such that $l_i \in \langle e \in C \text{ $\vert$ } e \text{ is an exterior edge} \rangle$, we say that $e_i$ is removable from $C$. If $G^\prime$ is obtained from
$G$ by a sequence of removals of interior edges from cycles, and $G^{\prime\prime}$ has the same matrix as $G^\prime$, we say that $G$ decomposes into $G^{\prime\prime}$.
If a cycle $C$ has no interior edges in $G^\prime$, we say that $C$ splits off from $G$.
\end{defn}

We represent the removal of an edge $e_i$ from a single cycle algebraically, by replacing the current value of $l_i$ with zero. Notice that when a cycle $C$ splits off, it does so with the removable edges deleted and their endpoints identified.

\begin{ex} Consider the following edge labeled diamond graph $(D_{3,3} , \alpha)$.
\begin{figure}[H]
\begin{center}
\scalebox{0.16}{\includegraphics{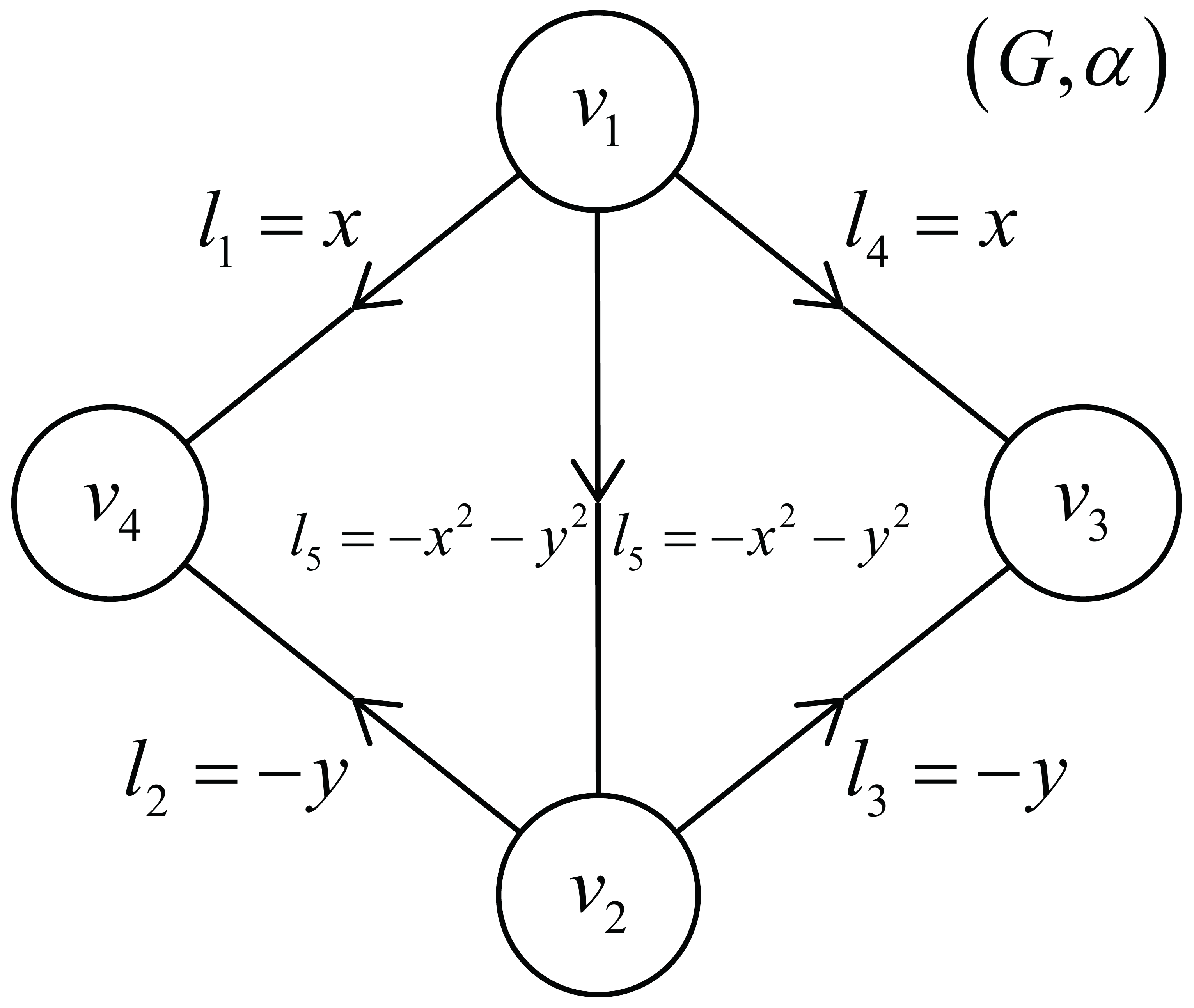}}
\caption{Edge labeled $D_{3,3}$}
\label{d331}
\end{center}
\end{figure}
\noindent The edge $l_5$ is an interior edge. It is removable from the right-side $3$-cycle of $(D_{3,3} , \alpha)$ since $l_5 = y l_3 - x l_4$. Hence we obtain the following graph $G^\prime$ by setting $l_5$ zero on the right-side $3$-cycle.
\begin{figure}[H]
\begin{center}
\scalebox{0.16}{\includegraphics{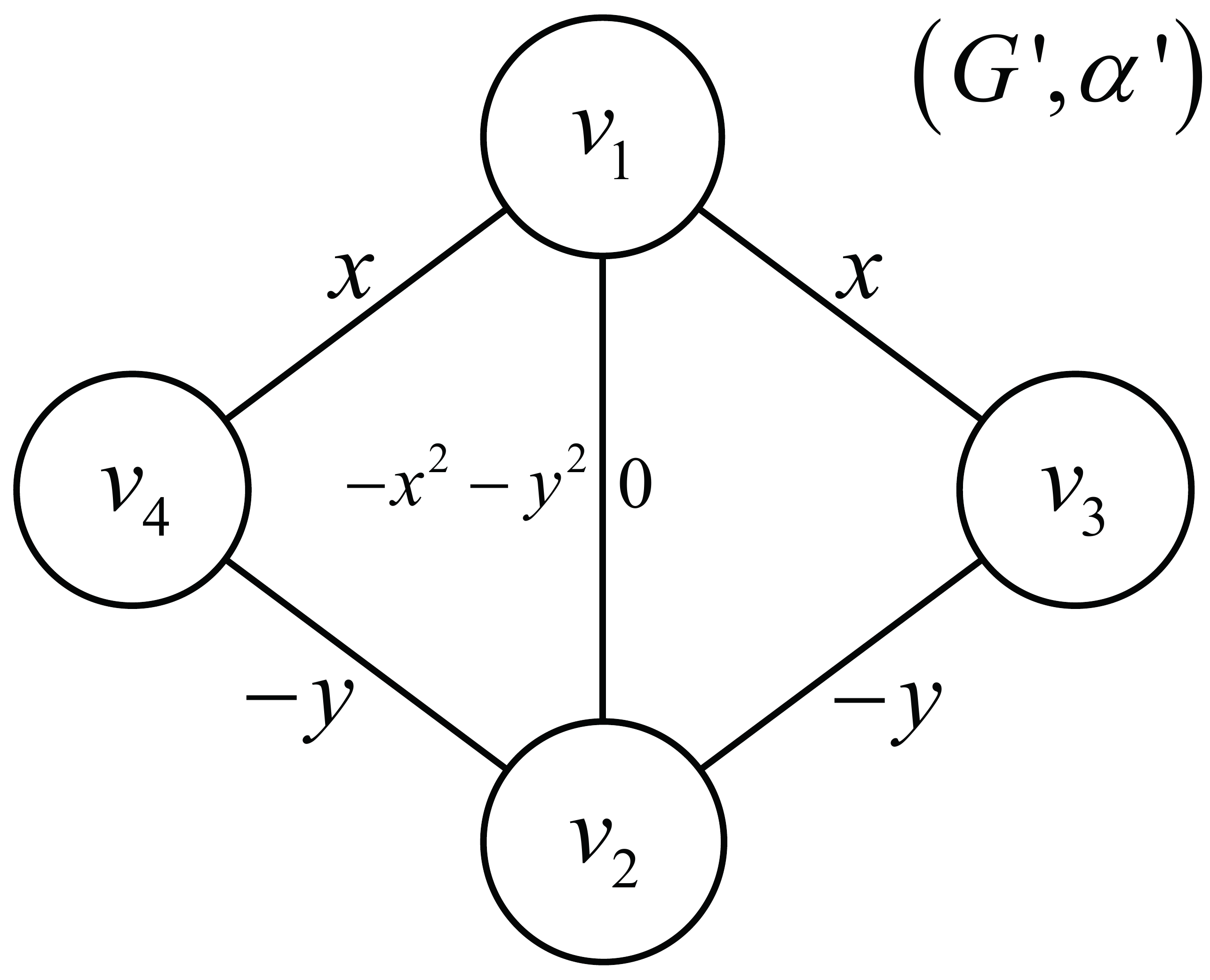}}
\caption{Removing an interior edge}
\label{d332}
\end{center}
\end{figure}
\noindent Finally we split off the right-side $3$-cycle as follows:
\begin{figure}[H]
\begin{center}
\scalebox{0.16}{\includegraphics{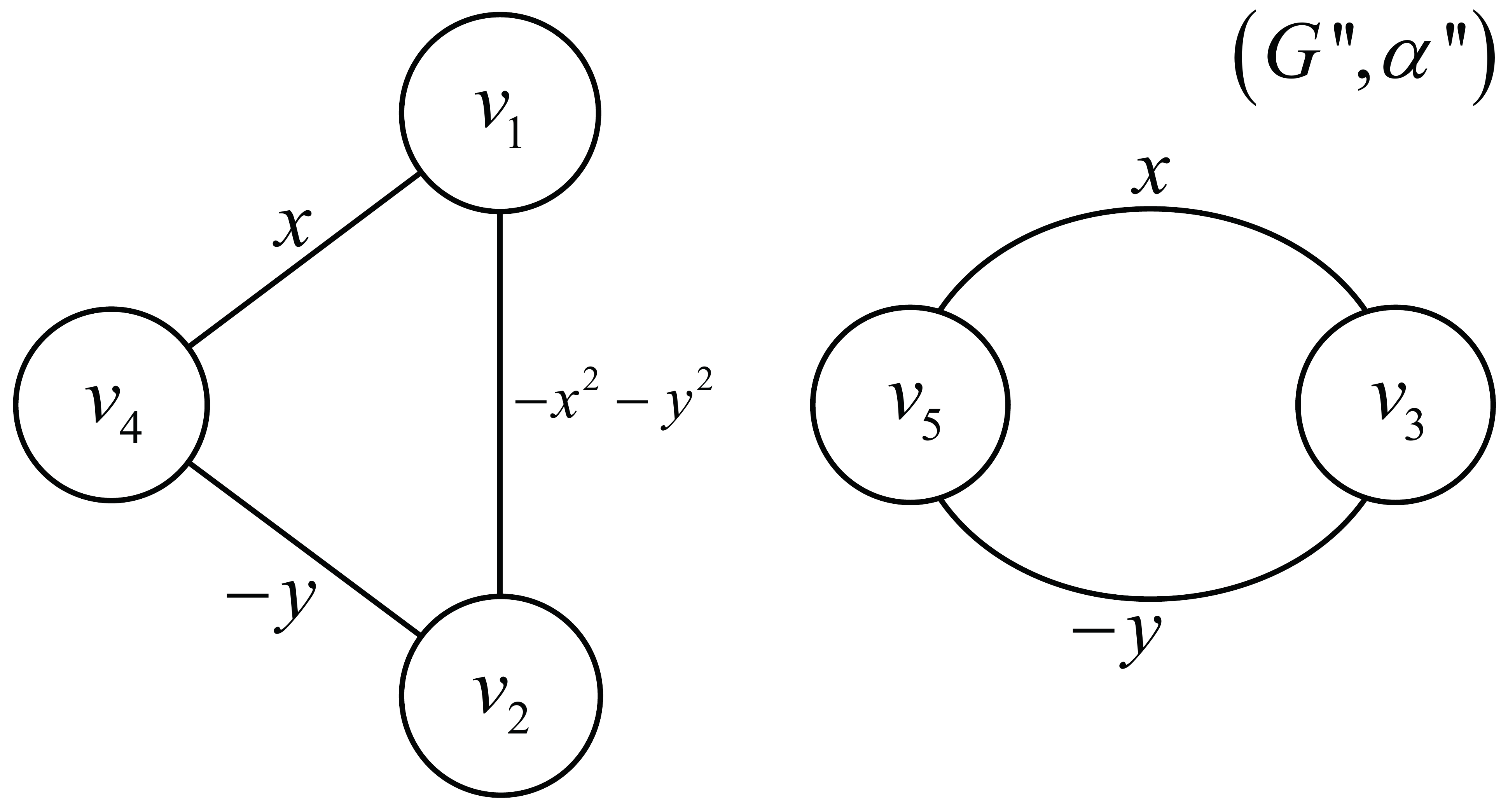}}
\caption{Splitting off cycles}
\label{d333}
\end{center}
\end{figure}
\label{exd33}
\end{ex}

The following theorem shows that the decomposition operations do not affect the isomorphism class of the module $B_{(G,\alpha)}$. Theorem~\ref{dectheo} and Corollary~\ref{nointerior} are proved in~\cite{Rose} for classical spline case. However, the proofs are identical for generalized splines.

\begin{theo} If $G$ decomposes into $G^{\prime\prime}$, then
\begin{displaymath}
B_{(G,\alpha)} \cong B_{(G^{\prime\prime},\alpha^{\prime\prime})}
\end{displaymath}
as $R$-modules.
\label{dectheo}
\end{theo}

\begin{proof} See Theorem 4.7 in~\cite{Rose}.
\end{proof}

\begin{ex} Consider the edge labeled graphs given in Figures~\ref{d331},~\ref{d332} and~\ref{d333}. Recall that the module $B_{(G , \alpha)}$ is represented as the kernel of the following matrix
\begin{displaymath}
A = \begin{pmatrix} x & -y & 0 & 0 & -x^2 - y^2 \\ 0 & 0 & -y & x & -x^2 - y^2
\end{pmatrix}.
\end{displaymath}
Since $B_{(G^{\prime} , \alpha^{\prime})}$ and $B_{(G^{\prime\prime} , \alpha^{\prime\prime})}$ are both given by the following matrix:
\begin{displaymath}
A^{\prime} = \begin{pmatrix} x & -y & 0 & 0 & -x^2 -y^2 \\ 0 & 0 & -y & x & 0
\end{pmatrix},
\end{displaymath} where the matrix $A^{\prime}$ corresponds to the graph $G^{\prime} $ obtained by removing the edge $e_5$ from the right-side  3-cycle of $(D_{3,3},\alpha)$ , simply setting the label of $e_5$ to be 0 in Figure~\ref{d331}, they are isomorphic. Now we consider
a map between $B_{(G , \alpha)}$ and $B_{(G^{\prime} , \alpha^{\prime})}$  defined by
\begin{align*}
\phi: B_{(G,\alpha)} &\to B_{(G^{\prime}, \alpha^{\prime})} \\
(r_1 , r_2 , r_3 , r_4 , r_5) &= (r_1 , r_2 , r_3 + yr_5 , r_4 - xr_5 , r_5).
\end{align*}
It can be easily observed that $\phi$ is an $R$-module isomorphism so that $B_{(G,\alpha)} \cong B_{(G^{\prime\prime} , \alpha^{\prime\prime})}$. Alternatively, $G$ decomposes into $G^{\prime}$ and finally into two disjoint cycles in $G^{\prime\prime}$. Hence by Theorem~\ref{dectheo}, we obtain the same result.
\label{syzygyex}
\end{ex}

\begin{cor} If $G^{\prime\prime}$ is a decomposition of $G$ into disjoint cycles $C_1 , \ldots , C_s$ and $p$ free edges, then
\begin{displaymath}
B_{(G,\alpha)} \cong \bigoplus \limits_{\substack{i=1}} ^s B_{(C_i,\alpha_i)} \oplus R^p.
\end{displaymath}
as $R$-modules.
\label{nointerior}
\end{cor}

\begin{proof} See Corollary 4.9 in~\cite{Rose}.
\end{proof}

Corolloary~\ref{nointerior} shows that, if $G$ has no interior edges then $B_{(G,\alpha)}$ is isomorphic to the direct sum of the syzygy modules of cycles in $G$ and the number of free edges copy of $R$.

In case all edge labels of graph $G$ are homogeneous polynomials of the same degree, each isomorphism in Theorem~\ref{mmodule}, Theorem~\ref{dectheo} and Corollary~\ref{nointerior} is a graded $R$-module isomorphism. In this case, we can talk about the Hilbert series of the syzygy module $B_{(G,\alpha)}$.

\begin{prop} Let $(G,\alpha)$ be an edge labeled graph such that all edge labels are homogeneous with the same degree $r$. Then $R_{(G,\alpha)} \cong R \oplus B_{(G,\alpha)}$ as graded $R$-modules with a degree shift in $B_{(G,\alpha)}$ of $r$.
\label{degshift}
\end{prop}

\begin{proof} See Theorem 2.2 in~\cite{Ros1}.
\end{proof}

\begin{prop} If $G$ decomposes into disjoint cycles $C_1 , \ldots , C_s$ and $p$ free edges, then
\begin{displaymath}
\mathcal{HS}(B_{(G,\alpha)}) = \mathcal{HS}(B_{(C_1,\alpha_1)}) + \cdots + \mathcal{HS}(B_{(C_s,\alpha_s)}) + \dfrac{p}{(1-t)^d}
\end{displaymath}
where $\mathcal{HS}(M)$ denotes the Hilbert series of $M$.
\label{hs}
\end{prop}

\begin{proof} See Corollary 4.9 in~\cite{Rose}.
\end{proof}

\begin{ex} The edge labeled graph $(G,\alpha)$ in Figure~\ref{d334} decomposes into disjoint cycles $C_1$ and $C_2$  in Figure~\ref{d335}.
\newpage
\begin{figure}
\centering
\begin{minipage}{.5\textwidth}
  \centering
  \scalebox{0.16}{\includegraphics{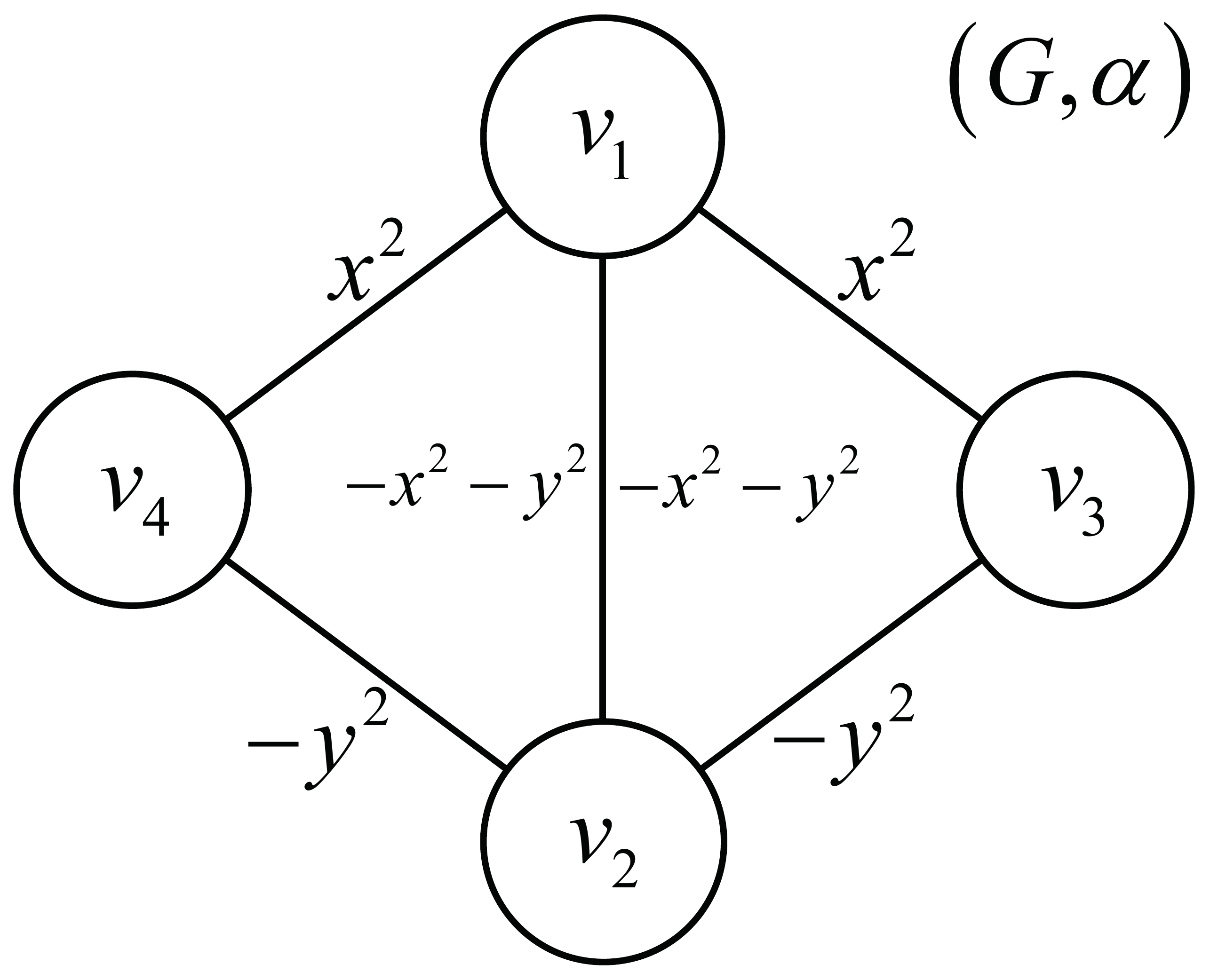}}
  \caption{Edge labeled $D_{3,3}$}
  \label{d334}
\end{minipage}%
\begin{minipage}{.5\textwidth}
  \centering
  \scalebox{0.16}{\includegraphics{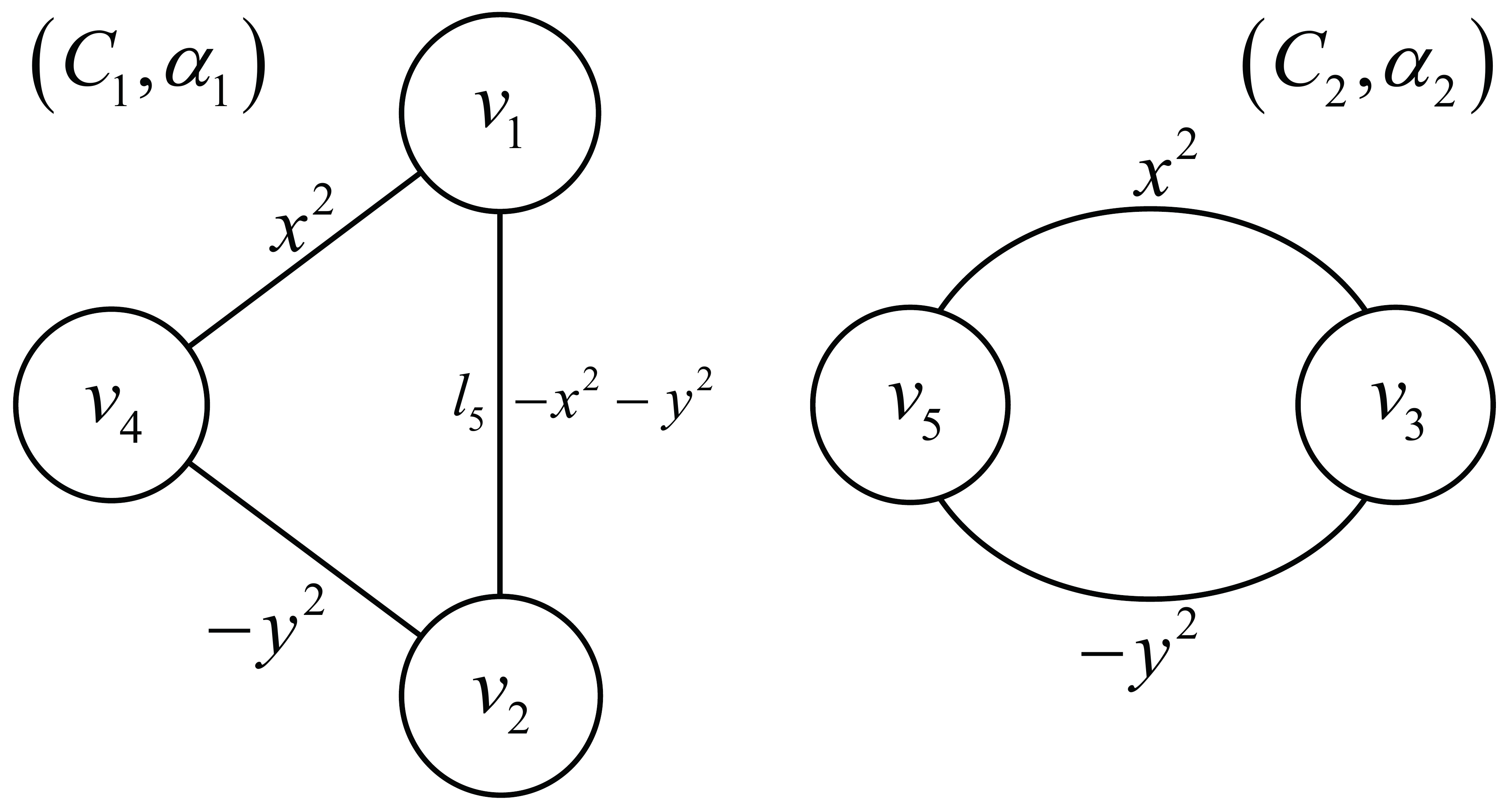}}
  \caption{Decomposition of $D_{3,3}$}
  \label{d335}
\end{minipage}
\end{figure}

\noindent Here we have 
\begin{align*}
B_{(G , \alpha)} &= \big\langle (y^2 , x^2 , 0,0,0) , (0,0,x^2,y^2,0) , (1,-1,-1,1,1) \big\rangle, \\ 
B_{(C_1 , \alpha_1)} &= \big\langle (y^2 , x^2 , 0) , (1,-1,1) \big\rangle, \\
B_{(C_2 , \alpha_2)} &= \big\langle (x^2 , y^2) \big\rangle. 
\end{align*}
It follows that the Hilbert series of these syzygy modules are given by
\begin{gather*}
\mathcal{HS}(B_{(G,\alpha)}) =\dfrac{1+2t^2}{(1-t)^2}, \quad
\mathcal{HS}(B_{(C_1,\alpha_1)}) =\dfrac{1+t^2}{(1-t)^2}, \quad
\mathcal{HS}(B_{(C_2,\alpha_2)}) =\dfrac{t^2}{(1-t)^2}.
\end{gather*} 
We conclude that $\mathcal{HS}(B_{(G,\alpha)}) = \mathcal{HS}(B_{(C_1,\alpha_1)}) + \mathcal{HS}(B_{(C_2,\alpha_2)})$, in agreement with  Proposition~\ref{hs}. In particular,
\begin{displaymath}
\mathcal{HS}(R_{(G,\alpha)}) = \dfrac{t^2 + 2t^4}{(1-t)^2} + \dfrac{1}{(1-t)^2}
\end{displaymath}
with a degree shift in $B_{(G,\alpha)}$ of $2$ by Proposition~\ref{degshift}.
\end{ex}

More details on graded generalized spline modules can be found in~\cite{Alt}.

\section{Freeness of the Spline Module $R_{(G,\alpha)}$}
\label{freeness}

Let the base ring $R = k[x_1 , \ldots , x_d]$ and $(C_n,\alpha)$ be an edge labeled cycle with edge labels $\{ l_1 , \ldots , l_n \}$. Fix the ideal $I = \langle l_1 , \ldots , l_n \rangle$. Then there exists a canonical short exact sequence as follows:
\begin{equation}
0 \to I \to R \to \nicefrac{R}{I} \to 0
\label{ses1}
\end{equation}
and by Theorem~\ref{pdim}, we obtain $\text{pd } I = \text{pd } \nicefrac{R}{I} - 1$ if $\text{pd } \nicefrac{R}{I} > 0$.

Another short exact sequence can be given as follows:
\begin{equation}
0 \to B_{(C_n , \alpha)} \xrightarrow{i} R^n \xrightarrow{f} I \to 0
\label{ses2}
\end{equation}
where $i$ is the inclusion map and $f : R^n \to I$ defined by $f(r_1 , \ldots , r_n) = \sum\limits_{\substack{t = 1}} ^n r_t l_t$. From this short exact sequence, we conclude that $\text{pd } B_{(C_n , \alpha)} = \text{pd } I - 1$ by Theorem~\ref{pdim} if $\text{pd } I > 0$. Combining two results obtained from (\ref{ses1}) and (\ref{ses2}), we have
\begin{equation}
\text{pd } B_{(C_n , \alpha)} = \text{pd } \nicefrac{R}{I} - 2.
\label{ses3}
\end{equation}
This observation leads us to the following result.

\begin{prop} Let $(C_n , \alpha)$ be an edge labeled cycle and the base ring $R$ be the bivariate polynomial ring $k[x,y]$. Then $R_{(C_n , \alpha)}$ is a free $R$-module.
\label{cycleprop}
\end{prop}

\begin{proof} Fix the ideal $I = \langle l_1 , \ldots l_n \rangle$ generated by the edge labels on $(C_n , \alpha)$. If $\text{pd } \nicefrac{R}{I} > 0$ and $\text{pd } I > 0$, then
\begin{displaymath}
\text{pd } B_{(C_n , \alpha)} = \text{pd } \nicefrac{R}{I} - 2 \leq 2-2 = 0
\end{displaymath}
since $\text{pd } \nicefrac{R}{I} \leq 2$ by the Hilbert Syzygy Theorem. Hence we conclude that $\text{pd } B_{(C_n , \alpha)} = 0$ and therefore $B_{(C_n , \alpha)}$ is a projective $R$-module. Since $R_{(C_n,\alpha)} \cong B_{(C_n,\alpha)} \oplus R$, the spline module $R_{(C_n,\alpha)}$ is also projective. By Quillen-Suslin Theorem, $R_{(C_n,\alpha)}$ is a free $R$-module.

If $\text{pd } \nicefrac{R}{I} = 0$, then $\nicefrac{R}{I}$ is a projective $R$-module and hence the short exact sequence~(\ref{ses1}) splits by Theorem~\ref{projmod} (a), which means $R \cong I \oplus \nicefrac{R}{I}$. Here $I$ is also projective $R$-module by Theorem~\ref{projmod2} and the short exact sequence~(\ref{ses2}) splits so that $R^n \cong B_{(C_n , \alpha)} \oplus I$. Therefore $B_{(C_n , \alpha)}$ is projective by Theorem~\ref{projmod2}. Since $R_{(C_n,\alpha)} \cong B_{(C_n,\alpha)} \oplus R$, the spline module $R_{(C_n,\alpha)}$ is also projective. By Quillen-Suslin Theorem, $R_{(C_n,\alpha)}$ is a free $R$-module.

If $\text{pd } I = 0$, then $I$ and $B_{(C_n , \alpha)}$ are projective $R$-modules as explained above. Hence $R_{(C_n,\alpha)}$ is also projective and so it is a free $R$-module.
\end{proof}

We can generalize Proposition~\ref{cycleprop} to graphs that contain only one cycle as follows:

\begin{cor} Let $(G , \alpha)$ be an edge labeled graph and $R$ be the bivariate polynomial ring $k[x,y]$. If $G$ contains only one cycle, then $R_{(G , \alpha)}$ is a free $R$-module.
\label{cor1}
\end{cor}

\begin{proof} Let $C_n \subset G$ be the only cycle contained in $G$ with edge labels $\{ l_1 , \ldots l_n \}$. Then $B_{(G , \alpha)} \cong B_{(C_n , \alpha^{\prime})} \oplus R^p$ by Corollary~\ref{nointerior} where $p$ is the number of the free edges of $G$. Here $B_{(C_n , \alpha^{\prime})}$ is projective by Proposition~\ref{cycleprop} and so is $B_{(G , \alpha)}$. Hence $R_{(G , \alpha)}$ is also projective and it is free.
\end{proof}

Another generalization of Proposition~\ref{cycleprop} can be given as follows:

\begin{cor} Let $(G , \alpha)$ be an edge labeled graph with no interior edges and $R$ be the bivariate polynomial ring $k[x,y]$. Then $R_{(G , \alpha)}$ is a free $R$-module.
\label{cor2}
\end{cor}

\begin{proof} If $G$ has no interior edge, then we can assume that $G$ consists of disjoint cycles $C_1 , \ldots , C_s$ and $p$ free edges. Then $B_{(G,\alpha)} \cong \bigoplus \limits_{\substack{i=1}} ^s B_{(C_i,\alpha_i)} \oplus R^p$ by Corollary~\ref{nointerior}. Here $B_{(C_i , \alpha_i)}$ is projective for all $i$ by Proposition~\ref{cycleprop} and hence $B_{(G,\alpha)}$ is also projective. Thus $R_{(G,\alpha)}$ is free.
\end{proof}

If $G$ has interior edges, then $R_{(G,\alpha)}$ may not be free even the base ring is $k[x,y]$. We consider the following example:

\begin{ex} Let $(G, \alpha)$ be the diamond graph as in the Figure~\ref{diamond33}.
\begin{figure}[H]
\begin{center}
\scalebox{0.16}{\includegraphics{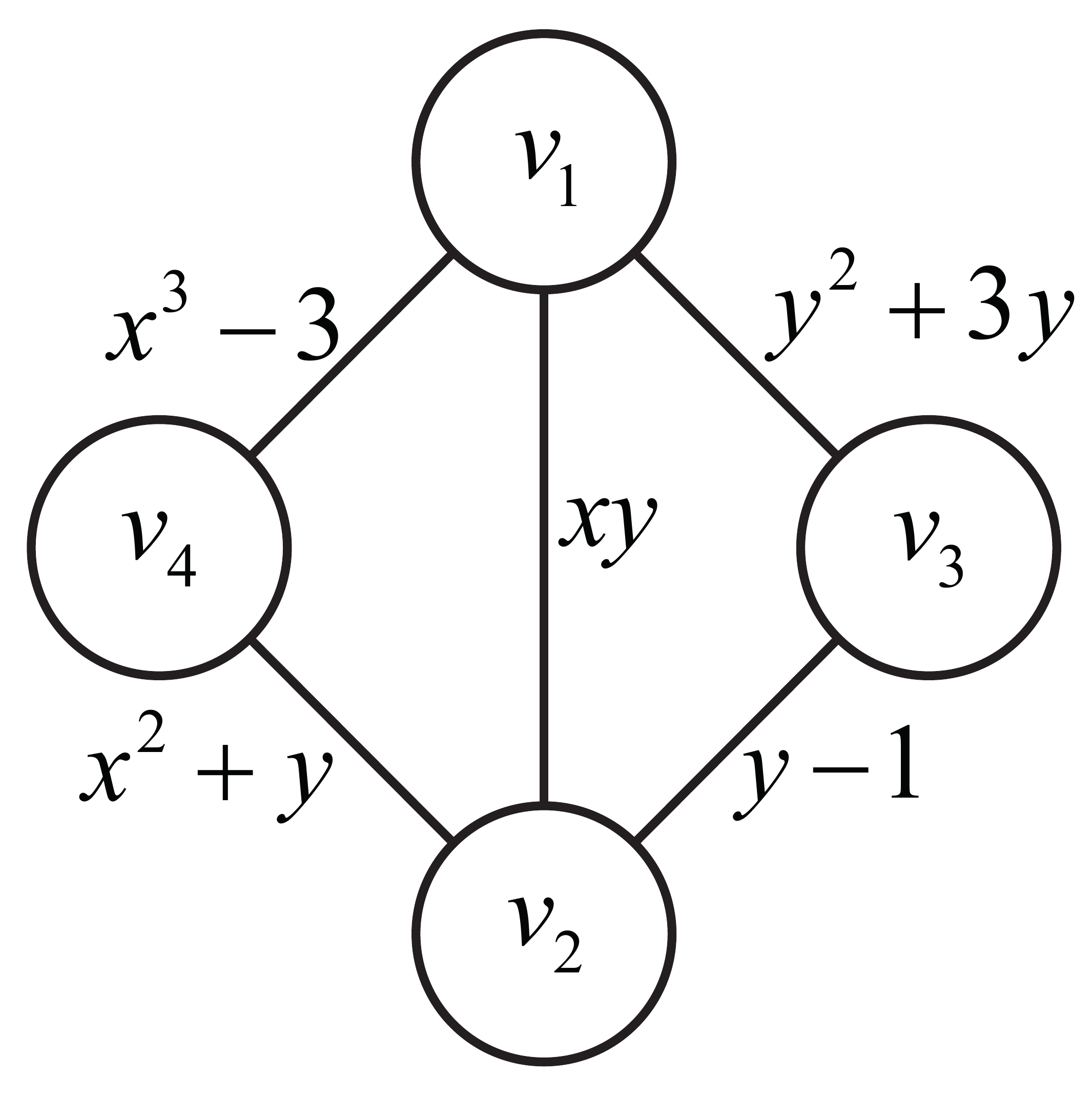}}
\caption{Edge labeled diamond graph}
\label{diamond33}
\end{center}
\end{figure}
\noindent
In this example $e_{12}$ is an interior edge which is not removable. Computations on CoCoA 4.7.5~\cite{Abb} show that the spline module $R_{(G,\alpha)}$ is not free.
\end{ex}

As a result of Corollary~\ref{nointerior}, we obtain the following outcome:

\begin{cor} Let $R=k[x,y]$. If  $(G , \alpha)$ decomposes into disjoint cycles with $p$ free edges, then $R_{(G , \alpha)}$ is a free $R$-module.
\label{cor3}
\end{cor}

\begin{ex} Consider the diamond graph $G$ in Figure~\ref{d331}, discussed in Example~\ref{exd33}. This graph decomposes into two disjoint cycles as in Figure~\ref{d333} and hence $R_{(G , \alpha)}$ is free by Corollary~\ref{cor3}.
\end{ex}

Proposition~\ref{cycleprop} holds also for cycles of rank at most  two on $R=k[x_1, \ldots , x_d]$. In this case, we use Theorem~\ref{rkthm} to prove the proposition below: 

\begin{prop} Let $(C_n , \alpha)$ be an edge labeled cycle and the base ring $R$ be the polynomial ring $k[x_1, \ldots , x_d]$. If $\emph{rk } C_n \leq 2$, then $R_{(C_n , \alpha)}$ is a free $R$-module.
\label{cycleprop2}
\end{prop}

\begin{proof} Fix the ideal $I = \langle l_1 , \ldots l_n \rangle$ generated by the edge labels on $(C_n , \alpha)$. Here $\text{pd } \nicefrac{R}{I} \leq \text{rk } C_n \leq 2$ by Theorem~\ref{rkthm} and thus $R_{(C_n , \alpha)}$ is free by Proposition~\ref{cycleprop}.
\end{proof}

As a result of Proposition~\ref{cycleprop2}, Corollaries~\ref{cor1},~\ref{cor2} and  ~\ref{cor3} can be generalized as follows.

\begin{cor} Let $R$ be the polynomial ring $k[x_1, \ldots , x_d]$.
\begin{enumerate}[label=\emph{(\alph*)}]
	\item If $G$ contains only one cycle $C_n$ and $\emph{rk } C_n \leq 2$, then $R_{(G , \alpha)}$ is a free $R$-module.
	\item If $G$ decomposes into disjoint cycles $C_1 , \ldots , C_s$ and $p$ free edges where $\emph{rk } C_i \leq 2$ for all $i$, then $R_{(G , \alpha)}$ is a free $R$-module.
\end{enumerate}
\label{cor4}
\end{cor}

The following example shows that  the converse of Corollary~\ref{cor4} is false.

\begin{ex} Consider the edge labeled $3$-cycle $G_1$ over $R=k[x,y,z]$ in Figure~\ref{rk1}. Let $I= \langle x^3 + yz^2 , x^2 + y^2 , xz^2 + y^3 \rangle$ be an ideal of $R$. By using Macaulay2, we obtain $\text{pd } \big( \nicefrac{R}{I} \big) = 2$. It follows that  the syzygy module is free by (\ref{ses3}) and thus $R_{(G_1,\alpha)}$ is free although $\text{rk } G_1 = 3$.
\label{rankex}
\end{ex}

\end{document}